\documentclass[11pt]{amsart} \textwidth=14.5cm \oddsidemargin=1cm
\usepackage{amsmath,wasysym}
\usepackage{amsxtra}
\usepackage{amscd}
\usepackage{amsthm}
\usepackage{amsfonts}
\usepackage{amssymb}
\usepackage{arydshln}
\usepackage{eucal}
\usepackage{graphics, color}
\usepackage{hyperref,mathrsfs}
\usepackage[usenames,dvipsnames]{xcolor}
\usepackage{tikz}
\usepackage{stmaryrd}
\usepackage[all]{xy}
\usepackage{hyperref}
\usepackage{color}
\usepackage{bbm} 
\allowdisplaybreaks

\hypersetup{colorlinks,linkcolor=blue, urlcolor=blue,
citecolor=blue}

\newtheorem{thm}{Theorem} [section]
\newtheorem{lem}[thm]{Lemma}
\newtheorem{cor}[thm]{Corollary}
\newtheorem{prop}[thm]{Proposition}

\newtheorem{rmk}[thm]{Remark}

\theoremstyle{definition}
\newtheorem{definition}[thm]{Definition}
\newtheorem{example}[thm]{Example}

\newcommand{\eps}{\varepsilon}
\newcommand{\A}{\mathcal A}
\newcommand{\bC}{\mathbb C}
\newcommand{\bN}{\mathbb N}
\newcommand{\bZ}{\mathbb Z}
\newcommand{\bQ}{\mathbb Q}

\newcommand{\lrs}[1]{\left[\!\!\left[{#1}\right]\!\!\right]}
\newcommand{\N}{\mathbb{N}}
\newcommand{\qS}{{\boldsymbol{\mathcal{S}}_q(n,r)}}
\newcommand{\qSK}{\boldsymbol{\mathcal{S}}_{q,K}(n,r)}
\newcommand{\qSL}{\boldsymbol{\mathcal{S}}_{q,L}(n,r)}
\newcommand{\qSj}{{\boldsymbol{\mathcal{S}}^\jmath(n,r)}}
\newcommand{\qSi}{{\boldsymbol{\mathcal{S}}^\imath(n,r)}}
\newcommand{\Q}{\mathbb{Q}}
\newcommand{\up}{\upsilon}
\newcommand{\vars}{\varsigma}
\newcommand{\Z}{\mathbb{Z}}
\newcommand{\tTheta}{\Theta^{\jmath}}
\newcommand{\tOmega}{\Omega^{\jmath}}
\newcommand{\tM}{\widetilde M}

\newcommand{\ro}{\operatorname{ro}}
\newcommand{\co}{\operatorname{co}}
\newcommand{\row}{\operatorname{row}}
\newcommand{\col}{\operatorname{col}}
\newcommand{\Mat}{\operatorname{Mat}}
\newcommand{\diag}{\operatorname{diag}}
\newcommand{\End}{\operatorname{End}}
\newcommand{\spanK}{\operatorname{Span}_{K}}
\newcommand{\rank}{\operatorname{rank}}
\newcommand{\rref}{\operatorname{rref}}
\newcommand{\Tr}{\operatorname{Tr}}
\newcommand{\Id}{\operatorname{Id}}
\numberwithin{equation}{section}

\title[Centers of $q$-Schur algebras]
{Centers of $q$-Schur algebras}

\author[Jian Chen]{Jian Chen}
\address{School of Mathematical Sciences,
	University of Science and Technology of China,
	Hefei 230026, China}
\email{JianChen96@outlook.com (Chen)}

\author[Xirui Yu]{Xirui Yu}
\address{School of Mathematical Sciences,
    East China Normal University, Shanghai 200241, China}
\email{52280155007@stu.ecnu.edu.cn (Yu)}

\begin{document}

\begin{abstract}

We give a direct BLM-coordinate description of the centers of generic $q$-Schur algebras for type $A$ and type $B$. For the $q$-Schur algebras $\qS$ of type $A$, we show that any central element is supported on some matrices satisfying certain balance properties, and that its coefficients are precisely the solutions of an explicit finite linear system derived from the left and right multiplication formulas for the Chevalley generators of $\qS$. 
This yields an effective reduced row echelon form construction of a $\Q(\upsilon)$-basis of the center, a basis that we then compare with the central bases obtained by Fu via quantum Schur-Weyl duality. 
For the $q$-Schur algebras $\qSj$ of type $B$, we first apply a method similar to that for type $A$ to provide a basis for the center. In addition, we provide another structural description of the center by decomposing $\qSj$ into the $q$-Schur algebras of type $A$, and these two descriptions are compatible. The center of a variant ${\boldsymbol{\mathcal{S}}^\imath(n,r)}$ of $\qSj$ is also discussed.

\end{abstract}

\maketitle
\setcounter{tocdepth}{2}
\tableofcontents
%

\section{Introduction}
\subsection{Background and motivation} 

The $q$-Schur algebra, a quantum analogue of the Schur algebra, is one of the standard intersections of Hecke algebras, quantum groups, and Schur-Weyl duality. In type $A$, it can be realized as an endomorphism algebra for permutation modules over the Hecke algebra of the symmetric group \cite{DJ89,Jim86}. Donkin's work \cite{Don98} develops its representation theory, while Jimbo's quantum Schur-Weyl duality explains its relationship with quantum enveloping algebras \cite{Jim86}. A geometric realization, due to Beilinson, Lusztig and MacPherson (abbr. BLM), constructs these $q$-Schur algebras through convolution algebras and gives a normalized basis indexed by some matrices \cite{BLM90}. 
We shall use this normalized BLM basis throughout, which is convenient for explicit calculations, since multiplication by the Chevalley generators obeys formulas that shift entries between adjacent rows or columns. The (type $A$) algebraic and geometric constructions are surveyed in \cite{DDPW08}, and related BLM-type methods have been developed in \cite{DF10,DDF12,BW18,BKLW18,FLLLW20,LL21,DW22,DW22b,FLLLW23,LY26}.

Nowadays there have been various generalizations of Schur algebras and $q$-Schur algebras in the literature. Donkin \cite{Do86,Do87} formulated a family of generalized Schur algebras associated with a reductive group $G$ of arbitrary type, whose quantum analogues were studied by Doty \cite{Dot03}. 
The theory of $q$-Schur algebras of type~$B$ has developed along several overlapping lines. A version of $q$-Schur algebras as centralizer algebras of type~$B$ Hecke modules was formulated independently by Dipper–James–Mathas and Du–Scott \cite{DJM98,DS00}. 
Green introduced the hyperoctahedral Schur algebra \cite{Gre97} as a type‑$B$ analogue, which later became part of the broader family of Schur-type algebras arising from quantum symmetric pairs. 
Bao, Kujawa, Li, and Wang supplied a BLM-type geometric Schur duality in classical type \cite{BKLW18}. Lai and Luo treated unequal parameters \cite{LL21}. Lai, Nakano, and Xiang then established the structure and representation theory for $q$-Schur algebras of type~$B$. 
Luo and Wang placed these constructions in a uniform framework for $q$-Schur algebras and Schur dualities of arbitrary finite type \cite{LW22}.

The center of a finite-dimensional algebra is a fundamental invariant. In the semisimple case, it records the block idempotents; in more general situations, it still measures a substantial part of the algebra's internal symmetry. Centers of Schur-type algebras have therefore been studied by several methods. Geetha and Prasad described the center of the classical Schur algebra via Schur-Weyl duality \cite{GP16}. In the quantum case, Fu proved that the natural Schur-Weyl maps preserve centers and used bases of the Hecke center to construct bases of the center of $q$-Schur algebras \cite{Fu21}. The center of $q$-Schur algebra $\mathbf{U}(2,r)$ in its normalized BLM basis is obtained in \cite{GL21} (see \cite{LSY25} for super case $\textup{S}(1|1,r)$). Recently, the center of the affine $q$-Schur algebra was determined by means of equivariant K-theory in \cite{LXY26}. These works motivate the problem considered here: to describe the center directly in normalized BLM basis for arbitrary $n$ and $r$.

\subsection{Main idea and results}

This paper takes that problem as its starting point. Since $q$-Schur algebras are generated by the Chevalley elements and the diagonal idempotents, it is enough to test centrality against these generators. The diagonal idempotents first force a support condition: a central element can be supported only on balanced matrices, namely matrices whose row sum vector equals their column sum vector (cf. Lemma \ref{lem:balanced-support} for $\qS$ \& Lemma \ref{lem:balancedB} for $\qSj$). 
Once this condition is imposed, the BLM multiplication formulas for the Chevalley generators turn the remaining commutator equations into a finite linear system (cf. \eqref{eq:matrix-system}).

After fixing an order on the balanced matrices and an order on the equations, the system has the form $Mx=0$. The map
$$x=(x_A)\longmapsto \sum_A x_A[A]$$
identifies $\ker M$ with the center (cf. \eqref{eq:kernel-center-isomorphism} \& \eqref{eq:RREFbasisB}). Reducing $M$ to reduced row echelon form (abrr. RREF) then gives a canonical basis once the chosen column order is fixed: each free column is set to one in turn, and all the other free variables are set to zero. We call the resulting central elements the RREF basis (cf. Theorem \ref{cor:rref-basis} \& \ref{cor:RREFbasisB}). This construction is algorithmic, while also providing a useful conceptual verification of the computed object.

We also compare the RREF basis with the central bases constructed by Fu from the center of the Hecke algebra \cite{Fu21}, which bases include variants related to the Geck-Rouquier basis \cite{GR97} and to Jones' relative norm basis \cite{Jo90}. Once the transported Hecke central elements are written in normalized BLM basis, the transition matrix from Fu's basis to our  RREF basis is obtained (cf. \eqref{eq:general-transition-matrix-G},  \eqref{eq:general-transition-matrix-J}, and Example \ref{ex:fu-32}). This gives a direct way to compare the two descriptions and makes clear where our RREF construction is simply a coordinate form of the same center, rather than a replacement for the Schur-Weyl approach. 

For type $B$, we carry out a similar linear-system construction for the $q$-Schur algebra $\qSj$. There is also an independent structural input: Lai, Nakano and Xiang proved that $\qSj$ decomposes as a direct sum of tensor products of $q$-Schur algebras of type $A$ \cite{LNX22}, accordingly, pulling back tensor products of type‑$A$ central bases yields another basis for the type‑$B$ center. (cf. Theorem \ref{thm:centerBviaIso}). The two constructions are compatible. 

The point of the paper is not to supersede the existing conceptual constructions of central bases. Rather, it supplies a transparent BLM basis realization of the center and an exact linear-algebraic procedure that is well suited for computation. This viewpoint also separates the generic statements over $\Q(\up)$ from the integral questions over $\Z[\up,\up^{-1}]$, which require additional arguments.

\subsection{Organization}

The paper is organized as follows. Section \ref{sec:Atype} treats type $A$. We recall the normalized BLM basis and the multiplication formulas for $\qS$ used in the paper, derive the central linear system, construct the RREF basis, and record the dimension and primitive central idempotent formulas. We then compare the RREF basis with central bases transported from the Hecke center. Section \ref{sec:Btype} treats type $B$. After recalling the $q$-Schur algebra of type $B$ and its multiplication formulas, we repeat the centrality argument. We then use the decomposition of $\qSj$ into type $A$ $q$-Schur algebras to obtain a tensor-product description of the center. The center of a variant ${\boldsymbol{\mathcal{S}}^\imath(n,r)}$ of $\qSj$ is discussed.

\subsection*{Notational conventions}

Throughout the paper, let $\upsilon$ be an indeterminate and put
$$ q=\upsilon^2, \qquad \A=\bZ[\upsilon,\upsilon^{-1}], \qquad K=\bQ(\upsilon).$$
The bar involution of $\A$ is an isomorphism  determined by $\overline \up=\up^{-1}$. For $a\in\bZ$, set
$$
\lrs{a}=\frac{\up^{2a}-1}{\up^{2}-1}\in\mathcal A.$$

\subsection*{Acknowledgment}
This paper is partially supported by the National Key R\&D Program of China (2024YFA1013802) and the Innovation Program for Quantum Science and Technology (2021ZD0302902).

\section{Center of $q$-Schur algebras of type $A$}\label{sec:Atype}
\subsection{The $q$-Schur algebras and multiplication formulas of type $A$}\label{sec:qschur-blm}
We first recall the definition of the $q$-Schur algebras and the multiplication formulas that will be used below. (cf. \cite{BLM90,DDPW08}).

Let $\mathcal H_{\mathcal A}(r)$ be the Hecke algebra of the symmetric group $\mathfrak S_r=\langle s_i\mid 1\leq i<r\rangle$ over $\A$. It is generated by $T_i~(1\leq i<r)$, subject to
$$T_i^2=(q-1)T_i+q,\quad
  T_iT_{i+1}T_i=T_{i+1}T_iT_{i+1}~(1\leq i<r-1),\quad
  T_iT_j=T_jT_i~(|i-j|>1).$$

Let $\N=\{0,1,2,\ldots\}$ and denote
\begin{align*}
    &\Lambda(n,r)=\{\lambda=(\lambda_1,\ldots,\lambda_n)\in\N^n\mid |\lambda|=\sum_i\lambda_i=r\},\\
    &\Lambda^{+}(n,r)=\{\lambda\in\Lambda(n,r)\mid \lambda_1\geq\cdots\geq\lambda_n\}.
\end{align*}
For $\lambda\in\Lambda(n,r)$, let $\mathfrak S_\lambda$ be the corresponding Young subgroup and put $x_\lambda=\sum_{w\in\mathfrak S_\lambda}T_w$. 
The integral $q$-Schur algebra of type $A$ over $\mathcal A$ is
$$\mathcal{S}_{q,\A}(n,r)=\End_{\mathcal H_{\mathcal A}(r)}
\Big(\bigoplus_{\lambda\in\Lambda(n,r)}x_\lambda\mathcal H_{\mathcal A}(r)\Big),$$
and let
$$\boldsymbol{\mathcal{S}}_{q,K}(n,r)=K\otimes_{\mathcal A}\mathcal{S}_{q,\A}(n,r).$$
We simply write ${\mathcal{S}}_{q}(n,r)={\mathcal{S}}_{q,\A}(n,r)$ and $\qS=\boldsymbol{\mathcal{S}}_{q,K}(n,r)$. 
Denote by $Z(\mathcal{S}_{q}(n,r))$ (resp., $Z(\qS)$) the center of the $q$-Schur algebra $\mathcal{S}_{q}(n,r)$ (resp., $\qS$).

Let
$$\Theta(n,r)=\Big\{A=(a_{ij})\in\Mat_n(\N)\mid |A|=\sum_{i,j}a_{ij}=r\Big\}.$$
For $A\in\Theta(n,r)$, we define its row/column sum vectors by
$$\ro(A)=\Big(\sum_j a_{1j},\ldots,\sum_j a_{nj}\Big),\qquad
  \co(A)=\Big(\sum_i a_{i1},\ldots,\sum_i a_{in}\Big).$$

In the following we collect some results from \cite{BLM90}; see also \cite{DDPW08}.
\begin{lem}\label{lem:blm-basis}
The algebra $\mathcal S_{q}(n,r)$ is a free $\mathcal A$-module with a normalized basis
$$\{[A]\mid A\in\Theta(n,r)\}.$$
There is an algebra anti-automorphism $\vars:\mathcal S_{q}(n,r)\to \mathcal S_{q}(n,r)$ such that $\vars([A])=[A^t]$, where $A^t$ is the transpose of $A$. Moreover, for $\lambda\in\Lambda(n,r)$ and $A\in\Theta(n,r)$, we have
$$[\diag(\lambda)][A]=
\begin{cases}
[A],& \lambda=\ro(A),\\
0,& \text{otherwise},
\end{cases}
\qquad
[A][\diag(\lambda)]=
\begin{cases}
[A],& \lambda=\co(A),\\
0,& \text{otherwise}.
\end{cases}$$
\end{lem}

Let $A\in\Mat_n(\N)$ have zero diagonal and let $\boldsymbol{j}=(j_1,\ldots,j_n)\in\Z^n$. Define
\begin{equation*}\label{eq:AjrA}
    A(\boldsymbol{j},r)=
\begin{cases}
\displaystyle\sum_{\lambda\in\Lambda(n,r-|A|)}\up^{\sum_i\lambda_i j_i}[A+\diag(\lambda)],& |A|\leq r,\\[6pt]
0,& |A|>r.
\end{cases}
\end{equation*}
For $1\leq h<n$ and $1\leq i\leq n$, write
$$e_h=E_{h,h+1}(\boldsymbol{0},r),\qquad
  f_h=E_{h+1,h}(\boldsymbol{0},r),\qquad
  d_i=O(\eps_i,r),$$
where $E_{ij}$ is the $n\times n$ matrix unit, $O$ is the zero matrix, and $\eps_i=(0,\cdots,0,\underset{(i)}{1},0,\cdots,0)$ is the $i$-th standard basis vector of $\Z^n$.
\begin{lem}\label{lem:generators}(\cite[Theorem 13.31]{DDPW08})
The $q$-Schur algebra $\qS$ is generated by $e_h$, $f_h$ for $1\leq h<n$ and by $d_i$ for $1\leq i\leq n$.
\end{lem}

For $A=(a_{ij})\in\Theta(n,r)$, $1\leq h< n$ and $1\leq p\leq n$, define
\begin{equation}\label{eq:betaB}
  \beta_p(A,h)=\sum_{j\geq p}a_{h,j}-\sum_{j>p}a_{h+1,j},
  \qquad
  \beta'_p(A,h)=\sum_{j\leq p}a_{h+1,j}-\sum_{j<p}a_{h,j}.
\end{equation}

By \cite[Lemma 3.4]{BLM90} and Lemma \ref{lem:blm-basis}, we have the following results. 

\begin{lem}\label{lem:blm-multiplication}
Let $A=(a_{ij})\in\Theta(n,r)$ and $1\leq h<n$. The following multiplication formulas hold in $\qS$,
\begin{align}
 e_h[A]
 &=\sum_{\substack{1\leq u\leq n\\ a_{h+1,u}>0}}
 L^+_{h,u}(A)[A+E_{h,u}-E_{h+1,u}],\label{eq:left-e}\\
 [A]e_h
 &=\sum_{\substack{1\leq u\leq n\\ a_{u,h}>0}}
 R^+_{h,u}(A)[A-E_{u,h}+E_{u,h+1}],\label{eq:right-e}\\
 f_h[A]
 &=\sum_{\substack{1\leq u\leq n\\ a_{h,u}>0}}
 L^-_{h,u}(A)[A-E_{h,u}+E_{h+1,u}],\label{eq:left-f}\\
 [A]f_h
 &=\sum_{\substack{1\leq u\leq n\\ a_{u,h+1}>0}}
 R^-_{h,u}(A)[A+E_{u,h}-E_{u,h+1}],\label{eq:right-f}
\end{align}
where
\begin{align*}
    &L^+_{h,u}(A)=\upsilon^{\beta_u(A,h)}\overline{\lrs{a_{h,u}+1}},
  \qquad
   L^-_{h,u}(A)=\upsilon^{\beta'_u(A,h)}\overline{\lrs{a_{h+1,u}+1}},\\
   & R^+_{h,u}(A)=\upsilon^{\beta'_u(A^t,h)}\overline{\lrs{a_{u,h+1}+1}},
  \qquad
   R^-_{h,u}(A)=\upsilon^{\beta_u(A^t,h)}\overline{\lrs{a_{u,h}+1}}.
\end{align*}
\end{lem}

\begin{proof}
The left multiplication formulas are the case $R=1$ of \cite[Lemma 3.4]{BLM90}.  The right multiplication formulas follow from the anti-involution in Lemma \ref{lem:blm-basis} because
$$[A]e_h=\varsigma(f_h[A^t]), \qquad [A]f_h=\varsigma(e_h[A^t]).$$
\end{proof}

\subsection{Centrality as an explicit linear system for type $A$}\label{sec:center-system}

Denote the set of \emph{balanced matrices} by
$$\Omega(n,r)=\{A\in\Theta(n,r)\mid \ro(A)=\co(A)\}.$$

\begin{lem}\label{lem:balanced-support}
If
$c=\sum_{A\in\Theta(n,r)}x_A[A]\in Z(\qS),$
then $x_A=0$ unless $A\in\Omega(n,r)$. Hence,
$$Z(\qS)\subseteq \spanK\{[A]\mid A\in\Omega(n,r)\}.$$
\end{lem}

\begin{proof}
For any $\lambda\in\Lambda(n,r)$, the equality $[\diag(\lambda)]c=c[\diag(\lambda)]$ and Lemma \ref{lem:blm-basis} give the assertion after comparing the coefficient of each $[A]$.
\end{proof}

For $1\leq h<n$, set $\alpha_h=\eps_h-\eps_{h+1}$ and define
\begin{align*}
    &\Theta^+_{n,r}(h)=\{B\in\Theta(n,r)\mid \ro(B)-\co(B)=\alpha_h\},\\
    &\Theta^-_{n,r}(h)=\{B\in\Theta(n,r)\mid \ro(B)-\co(B)=-\alpha_h\}.
\end{align*}
For $x=(x_A)_{A\in\Omega(n,r)}$, put
$$c(x)=\sum_{A\in\Omega(n,r)}x_A[A].$$

\begin{definition}[Center equations for type $A$]\label{def:center-equations}
For all $1\leq h<n$ and $B=(b_{ij})\in\Theta^+_{n,r}(h)$, impose the \emph{$e_h$-center equations}
\begin{align}
&\sum_{\substack{1\leq u\leq n\\ b_{h,u}>0}}
 x_{B-E_{h,u}+E_{h+1,u}}
 L^+_{h,u}(B-E_{h,u}+E_{h+1,u})  \notag\\
&\hspace{2cm} =
\sum_{\substack{1\leq u\leq n\\ b_{u,h+1}>0}}
 x_{B+E_{u,h}-E_{u,h+1}}
 R^+_{h,u}(B+E_{u,h}-E_{u,h+1}).
\label{eq:center-e}
\end{align}
For all $1\leq h<n$ and $B=(b_{ij})\in\Theta^-_{n,r}(h)$, impose the \emph{$f_h$-center equations}
\begin{align}\label{eq:center-f}
    &\sum_{\substack{1\leq u\leq n\\ b_{h+1,u}>0}}
 x_{B+E_{h,u}-E_{h+1,u}}
 L^-_{h,u}(B+E_{h,u}-E_{h+1,u}) \notag\\
&\hspace{2cm} =
\sum_{\substack{1\leq u\leq n\\ b_{u,h}>0}}
 x_{B-E_{u,h}+E_{u,h+1}}
 R^-_{h,u}(B-E_{u,h}+E_{u,h+1}).
\end{align}
If the subscript matrix of some $x$ is not in $\Omega(n,r)$, that term is interpreted as zero.  Under the displayed non-negativity conditions and the assumption $B\in\Theta^{\pm}_{n,r}(h)$, the subscript matrices that occur are in fact balanced.
\end{definition}

\begin{thm}\label{thm:center-system}
Let $x=(x_A)_{A\in\Omega(n,r)}$. Then
$$c(x)=\sum_{A\in\Omega(n,r)}x_A[A]$$
lies in $Z(\qS)$ if and only if $x$ satisfies all equations \eqref{eq:center-e} and \eqref{eq:center-f} for $1\leq h<n$.
\end{thm}

\begin{proof}
Assume first that $c(x)$ is central and fix $h$. By \eqref{eq:left-e}, every basis element occurring in $e_h[A]$ has the form $A^+_{h,u}:=A+E_{h,u}-E_{h+1,u}$ with $\ro(A^+_{h,u})-\co(A^+_{h,u})=\alpha_h$ since $A$ is balanced. By \eqref{eq:right-e}, every basis element occurring in $[A]e_h$ has the form $A^-_{u,h}:=A-E_{u,h}+E_{u,h+1}$ with $\ro(A^-_{u,h})-\co(A^-_{u,h})=\alpha_h$. Thus the nonzero coefficients of $e_hc(x)$ and $c(x)e_h$ can occur only on basis vectors indexed by $\Theta^+_{n,r}(h)$.

Take $B\in\Theta^+_{n,r}(h)$. In the expansion of $e_hc(x)$, the coefficient of $[B]$ receives contributions from $A=B-E_{h,u}+E_{h+1,u}$, and the nonnegativity condition is equivalent to $b_{h,u}>0$. These contributions form the left side of \eqref{eq:center-e}. In $c(x)e_h$, the coefficient of $[B]$ receives contributions from $A=B+E_{u,h}-E_{u,h+1}$, and the nonnegativity condition is equivalent to $b_{u,h+1}>0$. These contributions form the right side of \eqref{eq:center-e}. Comparing coefficients gives \eqref{eq:center-e}. A  similar argument applied to \eqref{eq:left-f} and \eqref{eq:right-f} gives \eqref{eq:center-f}.

Conversely, assume that all equations \eqref{eq:center-e} and \eqref{eq:center-f} hold. The coefficient comparisons described above show that $c(x)$ commutes with every $e_h$ and $f_h$. Since each $[A]$ in the expansion of $c(x)$ is balanced, Lemma \ref{lem:blm-basis} shows that $c(x)$ commutes with every $[\diag(\lambda)]$, and hence with every $d_i$. Lemma \ref{lem:generators} now implies that $c(x)$ is central.
\end{proof}

\begin{example}\label{ex:31}
Let $(n,r)=(3,1)$. The balanced matrices are
$$D_1=\diag(1,0,0),\qquad D_2=\diag(0,1,0),\qquad D_3=\diag(0,0,1).$$
For
$$c(x)=x_{D_1}[D_1]+x_{D_2}[D_2]+x_{D_3}[D_3],$$
the equation with $h=1$ and $B=E_{12}$ gives $x_{D_2}=x_{D_1}$, while the equation with $h=2$ gives $x_{D_3}=x_{D_2}$. Therefore
$$Z(\boldsymbol{\mathcal S}_q(3,1))=K([D_1]+[D_2]+[D_3])=K\cdot 1.$$
\end{example}

\subsection{Reduced row echelon form basis}

Arrange equations \eqref{eq:center-e} and \eqref{eq:center-f} in any fixed order. This gives a matrix equation
\begin{equation}\label{eq:matrix-system}
    M_{n,r}x=0, 
\end{equation}
where the columns of $M_{n,r}=(m_{ij})$ are labelled by $\Omega(n,r)$ and all entries $m_{ij}\in K$. Theorem \ref{thm:center-system} gives a $K$-linear isomorphism
\begin{equation}\label{eq:kernel-center-isomorphism}
    \Phi_{n,r}:\ker_K M_{n,r}\longrightarrow Z(\qS),
\qquad
x=(x_A)_{A\in\Omega(n,r)}\longmapsto\sum_{A\in\Omega(n,r)}x_A[A].
\end{equation}

We now transform $M_{n,r}$ into the \emph{reduced row echelon form}:
$$R:=\rref(M_{n,r}).$$

Let $\mathfrak{n}=|\Omega(n,r)|$ and fix an ordering
$$\Omega(n,r)=\{A_1,\ldots,A_\mathfrak{n}\}.$$
Write $x=(x_{A_1},\ldots,x_{A_\mathfrak{n}})^t\in K^\mathfrak{n}$. The $j$-th column of $R$ is labelled by $A_j$. Let
$$J_{\operatorname{piv}}=\{j\mid \text{the }\text{$j$-th column of }R\text{ contains the leading }1\text{ of a nonzero row}\}$$
and put $J_{\operatorname{free}}=\{1,\ldots,\mathfrak{n}\}\setminus J_{\operatorname{piv}}$. 
Define the pivot and free label sets by
$$ \mathcal P_{n,r}=\{A_j\mid j\in J_{\operatorname{piv}}\},
  \qquad
  \mathcal F_{n,r}=\{A_j\mid j\in J_{\operatorname{free}}\}.$$
Thus $\Omega(n,r)=\mathcal P_{n,r}\sqcup\mathcal F_{n,r}$. Note that these are sets of matrix labels, not sets of column vectors. 
For each $P\in\mathcal P_{n,r}$, let $\rho(P)$ denote the row of $R$ that contains the pivot in the column labelled by $P$. Then $Rx=0$ is equivalent to
\begin{equation}\label{eq:pivot-free-relation}
    x_P=\sum_{F\in\mathcal F_{n,r}}R'_{\rho(P),F}x_F
  \qquad P\in\mathcal P_{n,r}.
\end{equation}
where $R'_{\rho(P),F}=-R_{\rho(P),F}$ and $R_{\rho(P),F}$ denotes the 
$({\rho(P),F})$-entry of the matrix $R$.

For every $A_0\in\mathcal F_{n,r}$, define
\begin{align}\label{eq:rref-center-basis-element}
    C^{(n,r)}_{A_0}=[A_0]+\sum_{A\in\mathcal{P}_{n,r}}R'_{\rho(A),A_0}[A].
\end{align}

\begin{thm}[RREF-basis]\label{cor:rref-basis}
Keep the notation above. The set
$$\{C_{A_0}^{(n,r)}\mid A_0\in\mathcal F_{n,r}\}$$
forms a $K$-basis of $Z(\qS)$.
\end{thm}
\begin{proof}
There is an invertible matrix $Y$ such that $R=YM_{n,r}$, and therefore
$$M_{n,r}x=0\quad\Longleftrightarrow\quad YM_{n,r}x=0\quad\Longleftrightarrow\quad Rx=0.$$
In reduced row echelon form, each pivot column $P\in\mathcal P_{n,r}$ has a unique pivot row $\rho(P)$, the pivot coefficient is $1$, and all other pivot-column entries in that row description vanish. Thus $Rx=0$ is precisely the system \eqref{eq:pivot-free-relation}. Fixing the free variables to be the standard vector $x_F=\delta_{F,A_0}$, these standard choices form a basis of $\ker_K M_{n,r}$, and \eqref{eq:kernel-center-isomorphism} sends that basis to the stated central basis. See Example \ref{ex:32-rref} below for details.
\end{proof}

\begin{example}\label{ex:32-rref}
Let $(n,r)=(3,2)$. Order the nine balanced matrices as follows:
$$\begin{gathered}
A_1=\begin{pmatrix}0&0&0\\0&0&0\\0&0&2\end{pmatrix},\quad
A_2=\begin{pmatrix}0&0&0\\0&0&1\\0&1&0\end{pmatrix},\quad
A_3=\begin{pmatrix}0&0&0\\0&1&0\\0&0&1\end{pmatrix},\\[6pt]
A_4=\begin{pmatrix}0&0&0\\0&2&0\\0&0&0\end{pmatrix},\quad
A_5=\begin{pmatrix}0&0&1\\0&0&0\\1&0&0\end{pmatrix},\quad
A_6=\begin{pmatrix}0&1&0\\1&0&0\\0&0&0\end{pmatrix},\\[6pt]
A_7=\begin{pmatrix}1&0&0\\0&0&0\\0&0&1\end{pmatrix},\quad
A_8=\begin{pmatrix}1&0&0\\0&1&0\\0&0&0\end{pmatrix},\quad
A_9=\begin{pmatrix}2&0&0\\0&0&0\\0&0&0\end{pmatrix}.
\end{gathered}$$
The matrix obtained from all central equations has the following nonzero rows in reduced row echelon form:
$$\begin{pmatrix}
1&0&0&0&0&0&0&0&-1\\
0&1&0&0&0&0&0&\up^{-1}&-\up^{-1}\\
0&0&1&0&0&0&0&-1&0\\
0&0&0&1&0&0&0&0&-1\\
0&0&0&0&1&0&0&\up^{-1}&-\up^{-1}\\
0&0&0&0&0&1&0&\up^{-1}&-\up^{-1}\\
0&0&0&0&0&0&1&-1&0
\end{pmatrix}.$$
The pivot columns are labelled by $A_1,\ldots,A_7$, while the free columns are labelled by $A_8$ and $A_9$. Reading the rows gives
$$\begin{aligned}
 x_1&=x_9,&
 x_2&=\up^{-1}(x_9-x_8),&
 x_3&=x_8,\\
 x_4&=x_9,&
 x_5&=\up^{-1}(x_9-x_8),&
 x_6&=\up^{-1}(x_9-x_8),&
 x_7&=x_8.
\end{aligned}$$
Putting $s=x_{A_8}$ and $t=x_{A_9}$, an arbitrary central element takes the form
$$
\begin{aligned}
 c(s,t)=&\ t([A_1]+[A_4]+[A_9])+s([A_3]+[A_7]+[A_8])\\
&\quad +\up^{-1}(t-s)([A_2]+[A_5]+[A_6]).
\end{aligned}
$$
The choices $(s,t)=(1,0)$ and $(s,t)=(0,1)$ give
\begin{align*}
    &C^{\rref}_{110}=[A_3]+[A_7]+[A_8]-\up^{-1}([A_2]+[A_5]+[A_6]),\\
    &C^{\rref}_{200}=[A_1]+[A_4]+[A_9]+\up^{-1}([A_2]+[A_5]+[A_6]).
\end{align*}
\end{example}

\begin{rmk}\label{rem:n2}
When $n=2$, the balanced condition $\ro(A)=\co(A)$ is equivalent to $a_{12}=a_{21}$. Hence every balanced matrix has the unique form
$$A_{s,t}=\begin{pmatrix} r-2s-t&s\\ s&t\end{pmatrix},
  \qquad 0\leq s\leq \left\lfloor r/2\right\rfloor,
  \qquad 0\leq t\leq r-2s.$$ 
For $h=1$, the matrices in $\Theta^+_{2,r}(1)$ have the form
$B_{s,t}=\begin{pmatrix} r+1-2s-t&s\\ s-1&t\end{pmatrix}$.  
Substituting $n=2$ into \eqref{eq:center-e} gives
\begin{align*}
    v^{s-1}\overline{\lrs{s}}x_{s-1,t}
 +v^{3s+2t-2-r}\overline{\lrs{t}}x_{s,t-1} =
 v^{s-1}\overline{\lrs{s}}x_{s-1,t+1}
 +v^{r-s-2t}\overline{\lrs{r+1-2s-t}}x_{s,t}.
\end{align*}
These are precisely the three types of equations in \cite[Lemma 3.3]{GL21}. In this situation, balanced matrices are symmetric, and the anti-automorphism exchanges the $e_1$ and $f_1$ equations without adding new conditions.

For $n\geq 3$, balanced does not imply symmetric. For instance, matrix 
$\begin{pmatrix}
0&1&0\\
0&0&1\\
1&0&0
\end{pmatrix}$ 
is balanced but not symmetric. Thus, one cannot check only the $e_h$-center equations in general. 
Hence, the general case should keep both the $e_h$ and $f_h$ center equations.
\end{rmk}

\begin{thm}\label{thm:center-dimension}
Over $K$, we have
$$\dim_K Z(\qS)=|\Lambda^{+}(n,r)|.$$
Equivalently, $\rank_K M_{n,r}=|\Omega(n,r)|-|\Lambda^{+}(n,r)|$.
\end{thm}
\begin{proof}
The generic Hecke algebra is split semisimple, and the $q$-Schur algebra $\qS$ is also split semisimple. Its simple modules are indexed by $\Lambda^{+}(n,r)$, which is a standard result for $q$-Schur algebras; see \cite{DJ89,DJ91,Jim86,Don98}. 
The dimension of the center of a finite-dimensional split semisimple algebra equals the number of isomorphism classes of simple modules. This gives the first formula. The second follows from Theorem \ref{thm:center-system}.
\end{proof}

\begin{rmk}\label{rem:integral}
The BLM multiplication structure constants in Lemma \ref{lem:blm-multiplication} lie in $\mathcal A$, so the same equations define an integral center system for $\mathcal S_{q}(n,r)$.  However, reduced row echelon form may require division by a nonzero Laurent polynomial.  Therefore Theorem \ref{cor:rref-basis} is a statement over $K$, and it should not be read as an unconditional $\A$-basis statement for $Z(\mathcal S_{q}(n,r))$.
\end{rmk}

\subsection{Primitive central idempotents}\label{sec:idempotents}
As a complement, in this section, we record the normalized BLM basis expression of primitive central idempotents for $\qS$.

\begin{prop}\label{prop:primitive-idempotents}
The algebra $\qS$ has a unique family of pairwise orthogonal primitive central idempotents
$$\{z_\lambda\mid \lambda\in \Lambda^{+}(n,r)\}$$
whose sum is $1$ and whose action on irreducible $\qS$-module $L(\mu)$ satisfies
$$z_\lambda|_{L(\mu)}=\delta_{\lambda,\mu}\Id_{L(\mu)}
  \qquad \text{for all }\mu\in \Lambda^{+}(n,r).$$
These elements form a $K$-basis of $Z(\qS)$.
\end{prop}
\begin{proof}
Since $\qS$ is split semisimple, the Wedderburn decomposition gives $$\qS\cong\prod_{\lambda\in \Lambda^{+}(n,r)}\End_K(L(\lambda)).$$  
 The identity in the $\lambda$-th matrix block and zero in all other blocks gives $z_\lambda$. The stated properties are immediate from this decomposition.
\end{proof}

For $\lambda\in \Lambda^{+}(n,r)$, let $\chi_\lambda$ denote the usual trace character of the simple module $L(\lambda)$. Define the semisimple trace
\begin{equation}\label{eq:semisimple-trace}
    \tau_0(x)=\sum_{\lambda\in \Lambda^{+}(n,r)}\Tr_{L(\lambda)}(x),
  \qquad x\in \qS. 
\end{equation}
Under the Wedderburn decomposition, $\tau_0$ is the sum of matrix traces over all simple blocks. Hence $(x,y)\mapsto\tau_0(xy)$ is a nondegenerate symmetric bilinear form. Let $[A]^\vee$ be the element dual to $[A]$ with respect to this form, so that
\begin{equation}\label{eq:dual-basis}
    \tau_0([A]^\vee[B])=\delta_{A,B}
  \qquad A,B\in\Theta(n,r). 
\end{equation}

\begin{thm}\label{thm:primitive-expansion}
For each $\lambda\in \Lambda^{+}(n,r)$, the primitive central idempotent of Proposition \ref{prop:primitive-idempotents} has the normalized BLM basis expansion
$$z_\lambda=\sum_{A\in\Omega(n,r)}\chi_\lambda([A]^\vee)[A]
.$$
\end{thm}
\begin{proof}
Write $z_\lambda=\sum_{A\in\Omega(n,r)}c_A[A]$ by Lemma \ref{lem:balanced-support}. By duality,
$$c_A=\tau_0([A]^\vee z_\lambda).$$
Since $z_\lambda$ acts on $L(\mu)$ as $\delta_{\lambda,\mu}\Id_{L(\mu)}$, we obtain
$$\begin{aligned}
 \tau_0([A]^\vee z_\lambda)
 =\sum_{\mu\in \Lambda^{+}(n,r)}\Tr_{L(\mu)}([A]^\vee z_\lambda)
 =\Tr_{L(\lambda)}([A]^\vee)
 =\chi_\lambda([A]^\vee).
\end{aligned}$$
This proves the expansion.
\end{proof}

\begin{example}[The case $(n,r)=(3,2)$]\label{ex:primitive-32}
Keep the notation in Example \ref{ex:32-rref}. Direct calculation gives
\begin{align*}
    \begin{array}{c|ccccccccc}
 &A_1&A_2&A_3&A_4&A_5&A_6&A_7&A_8&A_9\\
\hline
\chi_{(2)}([A_i]^\vee)
&1&\dfrac{\upsilon}{\upsilon^2+1}&\dfrac1{\upsilon^2+1}&1&\dfrac{\upsilon}{\upsilon^2+1}&\dfrac{\upsilon}{\upsilon^2+1}&\dfrac1{\upsilon^2+1}&\dfrac1{\upsilon^2+1}&1\\[0.8em]
\chi_{(1,1)}([A_i]^\vee)
&0&-\dfrac{\upsilon}{\upsilon^2+1}&\dfrac{\upsilon^2}{\upsilon^2+1}&0&-\dfrac{\upsilon}{\upsilon^2+1}&-\dfrac{\upsilon}{\upsilon^2+1}&\dfrac{\upsilon^2}{\upsilon^2+1}&\dfrac{\upsilon^2}{\upsilon^2+1}&0.
\end{array}
\end{align*}

Then, we have
\begin{align*}
 z_{(1,1,0)}
 &=\frac{\upsilon^2}{\upsilon^2+1}\bigl([A_3]+[A_7]+[A_8]\bigr)
   -\frac{\upsilon}{\upsilon^2+1}\bigl([A_2]+[A_5]+[A_6]\bigr),\\[0.4em]
 z_{(2,0,0)}
 &=[A_1]+[A_4]+[A_9]
   +\frac1{\upsilon^2+1}\bigl([A_3]+[A_7]+[A_8]\bigr)
   +\frac{\upsilon}{\upsilon^2+1}\bigl([A_2]+[A_5]+[A_6]\bigr).
\end{align*}
Moreover,
\begin{align*}
    z_{(1,1,0)}=\frac{\up^2}{\up^2+1}C^{{\rref}}_{110},\qquad z_{(2,0,0)}
  =C^{{\rref}}_{200}+\frac{1}{\up^2+1}C^{{\rref}}_{110}.
\end{align*}
\end{example}

\subsection{RREF basis and Fu's  central bases}\label{subsec:fu-rref-comparison}

In \cite{Fu21}, Fu constructs bases of the center by transporting known Geck-Rouquier and Jones bases of the Hecke center through quantum Schur--Weyl duality. In this subsection, we compare the basis elements constructed by Fu with our RREF-basis elements.

Before the comparison, we make explicit the field conventions. Fu works over $\bC(\up)$, whereas this paper works over $K=\bQ(\upsilon)~(q=\up^2)$. In this subsection we set
$$L=\bC(\upsilon),\qquad \qSL=L\otimes_K \qSK,
  \qquad \boldsymbol{\mathcal{H}}_L(r)=L\otimes_{\A} \mathcal{H}_{\A}(r).$$

The following is a standard result for finite-dimensional algebras (cf. \cite[\S 4.2]{DK94}).
\begin{lem}\label{lem:centerDirectTensor}
Let $A$, $B$ and $A_i$ be finite-dimensional $K$-algebras. We have
$$Z\Big(\bigoplus_i A_i\Big)=\bigoplus_i Z(A_i),
  \qquad
  Z(A\otimes_K B)=Z(A)\otimes_K Z(B).$$
\end{lem}
Then 
$$Z(\qSL)=L\otimes_K Z(\qSK).$$

Recall Fu's results in \cite{Fu21}: let $U_{\up}(\mathfrak{sl}_n)$ be the quantum group of $\mathfrak{sl}_n$ over $L$, $\Omega_{\up}$ be the natural representation of $U_{\up}(\mathfrak{sl}_n)$ and
$$\eta_{r,L}:\boldsymbol{\mathcal{H}}_L(r)^{\mathrm{op}}\longrightarrow \End_{U_{\up}(\mathfrak{sl}_n)}(\Omega_{\up}^{\otimes r})$$
be the map induced by the right Hecke action. Fu proves that
\begin{equation}\label{eq:Fu-center-preserving}
  Z(\qSL)=\eta_{r,L}\bigl(Z(\boldsymbol{\mathcal{H}}_L(r))^{\mathrm{op}})\bigr),
\end{equation}
and uses two Hecke central bases to obtain bases of $Z(\qSL)$ \cite[Theorems 3.10, 4.8 and 4.9]{Fu21}. 
For $\lambda\in \Lambda^{+}(n,r)$, set
$$G_\lambda^L:=\eta_{r,L}(f_\lambda^*),\qquad
  J_\lambda^L:=\eta_{r,L}(\mathcal N_\lambda(F_\lambda)).$$
Here $\{f_\lambda^*\}$ is the Geck--Rouquier basis of the Hecke center, and $\{\mathcal N_\lambda(F_\lambda)\}$ is the Jones relative norm basis \cite{GR97,Jo90}.  Fu's Theorems 4.8 and 4.9 state that
$$\mathcal G_{n,r}=\{G^L_\lambda\mid \lambda\in \Lambda^{+}(n,r)\},
  \qquad
  \mathcal J_{n,r}=\{J^L_\lambda\mid \lambda\in \Lambda^{+}(n,r)\}$$
are $L$-bases of $Z(\qSL)$.

Let
$$\mathcal C_{n,r}^{\rref}=\{C_{A_0}^{(n,r)}\mid A_0\in\mathcal F_{n,r}\}$$
be the RREF basis from Theorem \ref{cor:rref-basis}, extended to $L$ as above.  
If Fu’s central elements have the following expansion in the normalized BLM basis over $L$,
$$G_\lambda^L=\sum_{A\in\Omega(n,r)}g_{\lambda,A}[A],
  \qquad
  J_\lambda^L=\sum_{A\in\Omega(n,r)}j_{\lambda,A}[A],
  \qquad g_{\lambda,A},j_{\lambda,A}\in L,$$
then its expansion in the RREF basis is obtained directly from the free-column coordinates:
\begin{equation}\label{eq:free-column-transition}
  G_\lambda^L=\sum_{A_0\in\mathcal F_{n,r}}g_{\lambda,A_0}C_{A_0}^{(n,r)},
  \qquad
  J_\lambda^L=\sum_{A_0\in\mathcal F_{n,r}}j_{\lambda,A_0}C_{A_0}^{(n,r)}.
\end{equation}

Therefore, if
$$\mathcal C_{n,r}^{\rref}=(C_{A_1}^{(n,r)},\ldots,C_{A_{\mathfrak{m}}}^{(n,r)}),
  \quad
  \mathcal G_{n,r}=(G_{\lambda^{(1)}}^L,\ldots,G_{\lambda^{(\mathfrak{m})}}^L),\quad \mathcal J_{n,r}=(J_{\lambda^{(1)}}^L,\ldots,J_{\lambda^{(\mathfrak{m})}}^L),$$
where $A_1,\ldots,A_{\mathfrak{m}}$ is an ordering of $\mathcal F_{n,r}$ and $\lambda^{(1)},\ldots,\lambda^{(\mathfrak{m})}$ is an ordering of $\Lambda^{+}(n,r)$, then
\begin{equation}\label{eq:general-transition-matrix-G}
  \mathcal G_{n,r}=\mathcal C_{n,r}^{\rref}\,T_{\mathcal C\to\mathcal G}^L,
  \qquad
  (T_{\mathcal C\to\mathcal G}^L)_{ij}=g_{\lambda^{(j)},A_i}.
\end{equation}
The Jones transition matrix is given in a similar way by
\begin{equation}\label{eq:general-transition-matrix-J}
  \mathcal J_{n,r}=\mathcal C_{n,r}^{\rref}\,T_{\mathcal C\to\mathcal J}^L,
  \qquad
  (T_{\mathcal C\to\mathcal J}^L)_{ij}=j_{\lambda^{(j)},A_i}.
\end{equation}
These matrices are invertible over $L$ because both ordered sets are bases of the same $L$-vector space. See Example \ref{ex:fu-32} below for details.

A similar transition can be computed through primitive central idempotents.  Let $z_\mu\in Z(\qSK)$ be the primitive central idempotent labelled by $\mu\in \Lambda^{+}(n,r)$, and put
$$z_\mu^L=1\otimes z_\mu\in Z(\qSL).$$
Let $\varepsilon_\mu^H$ denote the primitive central idempotent of $\boldsymbol{\mathcal{H}}_L(r)$ labelled by $\mu$. If
\begin{equation}\label{eq:spectral-fu-general}
  f_\lambda^*=\sum_{\mu\in \Lambda^{+}(n,r)}a_{\lambda\mu}\varepsilon_\mu^H,
  \qquad
  \mathcal{N}_\lambda(F_\lambda)=\sum_{\mu\in \Lambda^{+}(n,r)}b_{\lambda\mu}\varepsilon_\mu^H,
  \qquad a_{\lambda\mu},b_{\lambda\mu}\in L,
\end{equation}
then Fu's center-preserving theorem gives
\begin{equation}\label{eq:spectral-schur-general}
  G_\lambda^L=\sum_{\mu\in \Lambda^{+}(n,r)}a_{\lambda\mu}z_\mu^L,
  \qquad
  J_\lambda^L=\sum_{\mu\in \Lambda^{+}(n,r)}b_{\lambda\mu}z_\mu^L.
\end{equation}

\begin{example}\label{ex:fu-32}
Let us work out the example $(n,r)=(3,2)$ in detail.  The partitions in $\Lambda^+(3,2)$ are $(2,0,0)$ and $(1,1,0)$.  We abbreviate them to $(2)$ and $(1,1)$.  Put $T=T_{s_1}$ in $\boldsymbol{\mathcal{H}}_L(2)$.  Then
$$T^2=(q-1)T+q,
  \qquad q=\upsilon^2.$$
The primitive central idempotents of $\boldsymbol{\mathcal{H}}_L(2)$ are
\begin{equation}\label{eq:eps-S2}
  \varepsilon_{(2)}^H=\frac{T+1}{q+1},
  \qquad
  \varepsilon_{(1,1)}^H=\frac{q-T}{q+1}.
\end{equation}

First consider the Geck--Rouquier basis.  The two conjugacy classes of $\mathfrak S_2$ are represented by $1$ and $s_1$.  Fu's formula for $f_\lambda^*$ gives
$$f_{(1,1)}^*=1,
  \qquad
  f_{(2)}^*=q^{-1}T.$$
Using \eqref{eq:eps-S2},
\begin{equation}\label{eq:GR-spectral-S2}
  f_{(1,1)}^*=\varepsilon^H_{(2)}+\varepsilon^H_{(1,1)},
  \qquad
  f_{(2)}^*=\varepsilon^H_{(2)}-q^{-1}\varepsilon^H_{(1,1)}.
\end{equation}
After applying $\eta_{2,L}$ to \eqref{eq:GR-spectral-S2}, we have
\begin{equation}\label{eq:G-spectral-32}
  G_{(1,1)}^L=z_{(2)}^L+z_{(1,1)}^L,
  \qquad
  G_{(2)}^L=z_{(2)}^L-q^{-1}z_{(1,1)}^L.
\end{equation}

Next consider the Jones basis.  For $\lambda=(2)$, the Young subgroup is all of $\mathfrak S_2$, so
$$\mathcal{N}_{(2)}(F_{(2)})=f_{(2)}^*=\varepsilon^H_{(2)}-q^{-1}\varepsilon^H_{(1,1)}.$$
For $\lambda=(1,1)$, the Young subgroup is trivial, then $F_{(1,1)}=1$.  Hence
$$\mathcal{N}_{(1,1)}(F_{(1,1)})=1+q^{-1}T^2=2+(1-q^{-1})T.$$
Its eigenvalues on the trivial and sign representations are $q+1$ and $(q+1)q^{-1}$.  Therefore
\begin{equation*}\label{eq:Jones-spectral-S2}
  \mathcal{N}_{(1,1)}(F_{(1,1)})=(q+1)\varepsilon^H_{(2)}+(1+q^{-1})\varepsilon^H_{(1,1)}.
\end{equation*}
Applying $\eta_{2,L}$ gives
\begin{equation}\label{eq:J-spectral-32}
  J_{(1,1)}^L=(q+1)z_{(2)}^L+(1+q^{-1})z_{(1,1)}^L,\qquad J_{(2)}^L=z_{(2)}^L-q^{-1}z_{(1,1)}^L.
\end{equation}

With the abbreviations $z_{(2)}=z_{(2,0,0)}$ and $z_{(1,1)}=z_{(1,1,0)}$, we insert the formulas from Example \ref{ex:primitive-32}, after scalar extension, we obtain, 
$$z_{(1,1)}^L=\frac{q}{q+1}C_{110}^{\rref},\qquad z_{(2)}^L=C_{200}^{\rref}+\frac{1}{q+1}C_{110}^{\rref}.$$
Substituting these into \eqref{eq:G-spectral-32} and \eqref{eq:J-spectral-32} gives
\begin{align*}
    &G_{(2)}^L=C_{200}^{\rref},
  \qquad
  G_{(1,1)}^L=C_{200}^{\rref}+C_{110}^{\rref},\\
  &J_{(2)}^L=C_{200}^{\rref},
  \qquad
  J_{(1,1)}^L=(q+1)C_{200}^{\rref}+2C_{110}^{\rref}.
\end{align*}
Finally, we write the transition matrix below:
\begin{equation}\label{eq:matrix-G-32}
  T^{L}_{\mathcal C\to\mathcal G}=\begin{pmatrix}1&1\\0&1\end{pmatrix},
  \qquad
  (T^{L}_{\mathcal C\to\mathcal G})^{-1}=\begin{pmatrix}1&-1\\0&1\end{pmatrix}.
\end{equation}
For the Jones basis,
\begin{equation}\label{eq:matrix-J-32}
  T^{L}_{\mathcal C\to\mathcal J}=\begin{pmatrix}1&q+1\\0&2\end{pmatrix}
  ,\qquad (T^{L}_{\mathcal C\to\mathcal J})^{-1}
  =\begin{pmatrix}1&-\dfrac{q+1}{2}\\[0.4em]0&\dfrac12\end{pmatrix}.
\end{equation}
\end{example}

\section{Centers of $q$-Schur algebras of type $B$}\label{sec:Btype}

We now turn to the $q$-Schur algebras of type $B$. Our notation follows Du and Wu \cite{DW22} (see also \cite{BKLW18,LL21}). Set
$$N=2n+1, \qquad m=n+1, \qquad i^*=N+1-i\quad \text{for }1\leq i\leq N.$$

\subsection{$q$-Schur algebras of type $B$  and their multiplication formulas} We first recall the definition of $q$-Schur algebras of type $B$. Let
$$\Lambda(n+1,r)=\{\lambda=(\lambda_1,\ldots,\lambda_n,\lambda_{n+1})\in\bN^{n+1}\mid |\lambda|=\sum_{i}\lambda_i=r\}.$$
For $\lambda\in\Lambda(n+1,r)$, 
define
$$ \widetilde\lambda=(\widetilde{\lambda}_i)_{1\le i\le N}=(\lambda_1,\ldots,\lambda_n,2\lambda_{n+1}+1,\lambda_n,\ldots,\lambda_1)\in\bN^N,$$
and let $\widetilde\Lambda(n+1,r)$ be the image of $\Lambda(n+1,r)$ under the map $\widetilde{~}$.  Thus $|\widetilde\lambda|=2r+1$.

Let $\mathcal H_{\A}(B_r)$ be the Hecke algebra over $\A$ associated with the Coxeter system $(W,S)$ of type $B_r$, where $S=\{s_1,\ldots,s_r\}$.  It is generated by $T_i=T_{s_i}~(1\le i\le r)$ and subject to the relations:
$$T_i^2=(q-1)T_i+q\quad \text{for all }i,
  \qquad
  T_iT_j=T_jT_i\quad \text{if } |i-j|>1,$$
$$T_iT_{i+1}T_i=T_{i+1}T_iT_{i+1}\quad \text{for }1\leq i<r-1,
  \qquad
  T_{r-1}T_rT_{r-1}T_r=T_rT_{r-1}T_rT_{r-1}.$$
  
For $\lambda\in\Lambda(n+1,r)$, let $W_{\lambda}$ be the parabolic subgroup generated by
$$S\setminus\{s_{\lambda_1+\cdots+\lambda_i}\mid 1\leq i\leq n\}.$$
Put $\widetilde x_{\lambda}=\sum_{w\in W_{\lambda}}T_w$.  The integral $q$-Schur algebra of type $B$ is
$$\mathcal{S}^\jmath(n,r)=\End_{\mathcal H_{\A}(B_r)}\left(\bigoplus_{\lambda\in\Lambda(n+1,r)}\widetilde{x}_{\lambda}\mathcal H_{\A}(B_r)\right),$$
and let $\qSj=K\otimes_{\A}\mathcal{S}^\jmath(n,r)$.  We write $Z(\qSj)$ for its center.

Denote
$$\tTheta(n,r)=\Big\{A=(a_{ij})\in\Mat_N(\bN)~\Big|~\sum_{i,j}a_{ij}=2r+1,
  \ a_{ij}=a_{i^*,j^*}\Big\}.$$
For $A\in\tTheta(n,r)$, we define its row sum vector $\row(A)=(\row(A)_i)_{i=1}^N$ and column sum vector $\col(A)=(\col(A)_j)_{j=1}^N$ by
$$ \row(A)_i=\sum_j a_{ij},
  \qquad
  \col(A)_j=\sum_i a_{ij}.$$

\begin{lem}\label{lem:BLMbasisB}
The algebra $\qSj$ is a free $K$-module with a normalized basis
$$\{[A]\mid A\in\tTheta(n,r)\}.$$
There is an algebra anti-involution $\widetilde\varsigma:\qSj \to\qSj$ satisfying $\widetilde\varsigma([A])=[A^t]$.  For $\lambda\in\Lambda(n+1,r)$ and $A\in\tTheta(n,r)$,
$$[\diag(\widetilde\lambda)][A]=
  \begin{cases}
    [A], & \widetilde\lambda=\row(A),\\
    0, & \text{otherwise},
  \end{cases}
  \qquad
  [A][\diag(\widetilde\lambda)]=
  \begin{cases}
    [A], & \widetilde\lambda=\col(A),\\
    0, & \text{otherwise}.
  \end{cases}$$
\end{lem}

Let $\widetilde E_{ij}\in \Mat_N(\bN)$ be the $N\times N$ matrix unit and define
$$ E^{\theta}_{i,j}=\widetilde E_{i,j}+\widetilde E_{i^*,j^*}=E^{\theta}_{i^*,j^*}.$$
In particular, $E^{\theta}_{m,m}=2\widetilde E_{m,m}$.  Let $\mathbf e_i$ be the $i$-th standard basis vector of $\bZ^N$ and set
$$\mathbf e_i^{\theta}=\mathbf e_i+\mathbf e_{i^*}=\row(E^{\theta}_{i,i}),
  \qquad
  \alpha_i^{\theta}=\mathbf e_i^{\theta}-\mathbf e_{i+1}^{\theta}.$$

For $\boldsymbol{j}=(j_1,\ldots,j_N)\in\bZ^N$ and $A=(a_{ij})\in\Mat_N(\bN)$ with  $a_{ii}=0$ and $a_{i,j}=a_{i^*,j^*}$, define the following elements of $\qSj$:
\begin{equation*}\label{eq:AjrB}
    A[\boldsymbol{j},r]=
\begin{cases}
\displaystyle\sum_{\lambda\in\Lambda(n+1,r-|A|/2)}
  \upsilon^{\sum_{i=1}^{N}\widetilde\lambda_i j_i}[A+\diag(\widetilde\lambda)],& |A|\leq 2r,\\[6pt]
0,& |A|>2r.
\end{cases}
\end{equation*}

\begin{lem}\label{lem:genB}(\cite[Corollary 3.13]{BKLW18})
The algebra $\qSj$ is generated by $e_h^\jmath:=E^{\theta}_{h,h+1}[\mathbf 0,r]$, $f_h^\jmath:=E^{\theta}_{h+1,h}[\mathbf 0,r]$ for $1\leq h\leq n$ and by $d_i^\jmath:=O[\mathbf e_i,r]$ for $1\leq i\leq n+1$.
\end{lem}

For $A=(a_{ij})\in\tTheta(n,r)$, $1\leq h\leq n$ and $1\leq p\leq N$, define
\begin{equation}\label{eq:betaB}
  \widetilde\beta_p(A,h)=\sum_{j\geq p}a_{h,j}-\sum_{j>p}a_{h+1,j}+\delta_{h,n}\delta^{\leq}_{p,n},
  \quad
  \widetilde\beta'_p(A,h)=\sum_{j\leq p}a_{h+1,j}-\sum_{j<p}a_{h,j},
\end{equation}
where $\delta^{\leq}_{a,b}=\begin{cases}1,& a\leq b,\\0,& a>b\end{cases}$ is a generalization of the Kronecker delta $\delta_{a,b}$.

\begin{prop}\label{prop:BLMmultB}
For $A\in\tTheta(n,r)$ and $1\leq h\leq n$, we have, in $\qSj$,
\begin{align}
  e_h^\jmath[A]
  &=\sum_{\substack{1\leq p\leq N\\ a_{h+1,p}\geq \epsilon^{\theta}_{h+1,p}}}
    \widetilde L^+_{h,p}(A)[A+E^{\theta}_{h,p}-E^{\theta}_{h+1,p}], \label{eq:BleftE}\\
  f_h^\jmath[A]
  &=\sum_{\substack{1\leq p\leq N\\ a_{h,p}\geq 1}}
    \widetilde L^-_{h,p}(A)[A-E^{\theta}_{h,p}+E^{\theta}_{h+1,p}], \label{eq:BleftF}\\
  [A]e_h^\jmath
  &=\sum_{\substack{1\leq p\leq N\\ a_{p,h}\geq 1}}
    \widetilde R^+_{h,p}(A)[A-E^{\theta}_{p,h}+E^{\theta}_{p,h+1}], \label{eq:BrightE}\\
  [A]f_h^\jmath
  &=\sum_{\substack{1\leq p\leq N\\ a_{p,h+1}\geq \epsilon^{\theta}_{p,h+1}}}
    \widetilde R^-_{h,p}(A)[A+E^{\theta}_{p,h}-E^{\theta}_{p,h+1}], \label{eq:BrightF}
\end{align}
where $\epsilon^{\theta}_{i,j}=
  \begin{cases}
    2, & (i,j)=(m,m),\\
    1, & \text{otherwise}.
  \end{cases}$ and 
\begin{align*}
    &\widetilde L^+_{h,p}(A)=\upsilon^{\widetilde\beta_p(A,h)}\overline{\lrs{a_{h,p}+1}},
  \qquad
  \widetilde L^-_{h,p}(A)=\upsilon^{\widetilde\beta'_p(A,h)}\overline{\lrs{a_{h+1,p}+1}},\\
  &\widetilde R^+_{h,p}(A)=\upsilon^{\widetilde\beta'_p(A^t,h)}\overline{\lrs{a_{p,h+1}+1}},
  \qquad
  \widetilde R^-_{h,p}(A)=\upsilon^{\widetilde\beta_p(A^t,h)}\overline{\lrs{a_{p,h}+1}}.
\end{align*}
\end{prop}
\begin{proof}
The formulas \eqref{eq:BleftE} and \eqref{eq:BleftF} are the multiplication formulas of \cite[Theorem 4.1]{DW22}.  The right multiplication formulas follow from the anti-involution $\widetilde\varsigma$:
$$[A]e_h^\jmath=\widetilde\varsigma(f_h^\jmath[A^t]),
  \qquad
  [A]f_h^\jmath=\widetilde\varsigma(e_h^\jmath[A^t]).$$
\end{proof}

\subsection{Centrality as an explicit linear system for type $B$}

In this subsection, we deduce some results for $Z(\qSj)$. We only list the statements here without proofs, as they can be proved by arguments similar to those for type 
$A$.

Define the balanced matrix set for type $B$:
$$\tOmega(n,r)=\{A\in\tTheta(n,r)\mid \row(A)=\col(A)\}.$$

\begin{lem}\label{lem:balancedB}
If $c(x)=\sum_{A\in\tTheta(n,r)}x_A[A]$ belongs to $Z(\qSj)$, then $x_A=0$ unless $A\in\tOmega(n,r)$.  Hence every central element has a unique expression
\begin{equation}\label{eq:centerExpansionB}
  c(x)=\sum_{A\in\tOmega(n,r)}x_A[A].
\end{equation}
\end{lem}

For $1\leq h\leq n$, define
\begin{align*}
    & \Theta^{\jmath,+}_{n,r}(h)=\{B\in\tTheta(n,r)\mid \row(B)-\col(B)=\alpha_h^{\theta}\},\\
    &\Theta^{\jmath,-}_{n,r}(h)=\{B\in\tTheta(n,r)\mid \row(B)-\col(B)=-\alpha_h^{\theta}\}.
\end{align*}

\begin{definition}[Center equations for type $B$]\label{def:centerEqB}
For $1\leq h\leq n$ and $B=(b_{ij})\in\Theta^{\jmath,+}_{n,r}(h)$, impose $e_h^\jmath$-center equations
\begin{align}
  &\sum_{\substack{1\leq p\leq N\\ b_{h,p}\geq 1}}
  x_{B-E^{\theta}_{h,p}+E^{\theta}_{h+1,p}}
  \widetilde L^+_{h,p}(B-E^{\theta}_{h,p}+E^{\theta}_{h+1,p}) \notag\\
  &\hspace{3cm}=
  \sum_{\substack{1\leq p\leq N\\ b_{p,h+1}\geq \epsilon^{\theta}_{p,h+1}}}
  x_{B+E^{\theta}_{p,h}-E^{\theta}_{p,h+1}}
  \widetilde R^+_{h,p}(B+E^{\theta}_{p,h}-E^{\theta}_{p,h+1}). \label{eq:centerE_B}
\end{align}
For $1\leq h\leq n$ and $B=(b_{ij})\in\Theta^{\jmath,-}_{n,r}(h)$, impose $f_h^\jmath$-center equations
\begin{align}
  &\sum_{\substack{1\leq p\leq N\\ b_{h+1,p}\geq \epsilon^{\theta}_{h+1,p}}}
  x_{B+E^{\theta}_{h,p}-E^{\theta}_{h+1,p}}
  \widetilde L^-_{h,p}(B+E^{\theta}_{h,p}-E^{\theta}_{h+1,p}) \notag\\
  &\hspace{3cm}=
  \sum_{\substack{1\leq p\leq N\\ b_{p,h}\geq 1}}
  x_{B-E^{\theta}_{p,h}+E^{\theta}_{p,h+1}}
  \widetilde R^-_{h,p}(B-E^{\theta}_{p,h}+E^{\theta}_{p,h+1}). \label{eq:centerF_B}
\end{align}
Again a term is interpreted as zero if its subscript matrix is not in $\tOmega(n,r)$.
\end{definition}

\begin{thm}\label{thm:centerSystemB}
Let $x=(x_A)_{A\in\tOmega(n,r)}$.  Then
$$c(x)=\sum_{A\in\tOmega(n,r)}x_A[A]$$
lies in $Z(\qSj)$ if and only if $x$ satisfies all equations \eqref{eq:centerE_B} and \eqref{eq:centerF_B} for $1\leq h\leq n$.
\end{thm}

Move all terms in \eqref{eq:centerE_B} and \eqref{eq:centerF_B} to the left.  Use the columns indexed by $\tOmega(n,r)$ and the rows indexed by
$$\{(h,B,+)\mid 1\leq h\leq n,
  \ B\in\Theta^{\jmath,+}_{n,r}(h)\}
  \sqcup
  \{(h,B,-)\mid 1\leq h\leq n,
  \ B\in\Theta^{\jmath,-}_{n,r}(h)\}.$$
This gives an explicit matrix $\tM_{n,r}$.  Its entries are determined by \eqref{eq:betaB}, \eqref{eq:centerE_B} and \eqref{eq:centerF_B}.

\begin{thm}\label{cor:RREFbasisB}
Over $K$,
\begin{align}\label{eq:RREFbasisB}
    \ker_K \tM_{n,r}\cong Z(\qSj),
  \qquad
  x\longmapsto \sum_{A\in\tOmega(n,r)}x_A[A].
\end{align}
After reducing $\tM_{n,r}$ to reduced row echelon form $\rref(\tM_{n,r})$, setting one free variable equal to $1$ and the other free variables equal to $0$ gives a central basis $\mathcal{C}^{\jmath}_{n,r}:=\{C^{\jmath}_F\mid F\in \mathcal{F}_{n,r}^{\jmath}\}$, where $\mathcal{F}_{n,r}^{\jmath}$ is the free label set of $\rref(\tM_{n,r})$.
\end{thm}

\subsection{Centers transported from type $A$}

Let
$$f_r^B(q)=\prod_{i=1-r}^{r-1}(q+q^{-i})\in K, \qquad q=\upsilon^2.$$
This element is nonzero in $K$ and hence invertible.  By Lai, Nakano and Xiang \cite[Section 3]{LNX22}, there is a $K$-algebra isomorphism
\begin{equation}\label{eq:LNXiso}
  \Psi:\qSj\xrightarrow{\ \sim\ }
  \bigoplus_{i=0}^{r}\boldsymbol{\mathcal{S}}_q(n+1,i)\otimes_K \boldsymbol{\mathcal{S}}_q(n,r-i).
\end{equation}

\begin{thm}\label{thm:centerBviaIso}
Over $K$, the isomorphism \eqref{eq:LNXiso} restricts to
\begin{equation}\label{eq:centerIsoB}
  Z(\qSj)\xrightarrow{\ \sim\ }
  \bigoplus_{i=0}^{r} Z(\boldsymbol{\mathcal S}_q(n+1,i))\otimes Z(\boldsymbol{\mathcal S}_q(n,r-i)).
\end{equation}
Moreover, recalling the RREF basis of type $A$ in Theorem \ref{cor:rref-basis}, the following set
\begin{equation}\label{eq:basisBviaIso}
  \bigcup_{i=0}^{r}
  \left\{
  \Psi^{-1}\left(C^{(n+1,i)}_{\alpha}\otimes C^{(n,r-i)}_{\beta}\right)
  \ \middle|\
  \alpha\in\mathcal F_{n+1,i},\ \beta\in\mathcal F_{n,r-i}
  \right\}
\end{equation}
forms a $K$-basis of $Z(\qSj)$. In particular,
\[
  \dim_K Z(\qSj)=
  \sum_{i=0}^{r}|\Lambda^+(n+1,i)|\,|\Lambda^+(n,r-i)|.
\]
\end{thm}

\begin{proof}
An algebra isomorphism maps centers isomorphically onto centers.  Applying Lemma \ref{lem:centerDirectTensor} to the right hand side of \eqref{eq:LNXiso} gives \eqref{eq:centerIsoB}. Theorem \ref{cor:rref-basis} gives a basis of each type $A$ center.  Tensoring these bases inside each direct summand and pulling them back by $\Psi^{-1}$ gives \eqref{eq:basisBviaIso}.  The dimension formula follows from Theorem \ref{thm:center-dimension}.
\end{proof}

\begin{cor}\label{cor:primitiveB}
The primitive central idempotents of $\qSj$ are parametrized by bipartitions $(\lambda,\mu)$ such that
\[
  |\lambda|+|\mu|=r,
  \qquad
  \ell(\lambda)\leq n+1,
  \qquad
  \ell(\mu)\leq n.
\]
They are
\begin{equation}\label{eq:primitiveB}
  z_{(\lambda,\mu)}=\Psi^{-1}(z_{\lambda}\otimes z_{\mu}),
\end{equation}
where $z_{\lambda}\in Z(\boldsymbol{\mathcal S}_q(n+1,|\lambda|))$ and $z_{\mu}\in Z(\boldsymbol{\mathcal S}_q(n,|\mu|))$ are the type $A$ primitive central idempotents in Proposition \ref{prop:primitive-idempotents}.
\end{cor}

\begin{cor}
    Over $K$, we have
    $$\ker_K \tM_{n,r}\cong
  \bigoplus_{i=0}^{r}\ker_K M_{n+1,i}\otimes\ker_K M_{n,r-i}.$$
\end{cor}
\begin{proof}
    It follows from Theorem \ref{cor:RREFbasisB}, \eqref{eq:kernel-center-isomorphism} and Theorem \ref{thm:centerBviaIso}.
\end{proof}
\begin{rmk} 
    There is an $\imath$-variant $\qSi$ of $\qSj$ (see \cite{BKLW18,LL21,DW22b,LNX22}). A similar linear system method can be applied to its center. We do not pursue this variant here. In addition, we have, 
    $$\qSi\cong e\qSj e,\quad\text{where}\quad e=\sum_{\lambda\in\Lambda(n+1,r)\,:\,\lambda_{n+1}=0}[\diag(\widetilde{\lambda})].$$
    (The explicit algebra isomorphism above is given in \cite{DW22b}). Then
    \begin{align}\label{eq:SieSje}
        Z(\qSi)\cong Z(e\qSj e)= eZ(\qSj)e.
    \end{align}
    We give here a short proof for the last equality in \eqref{eq:SieSje}. Since $\qSj$ is a finite-dimensional split semisimple $K$-algebra, we have, by the Wedderburn decomposition,
    $$\qSj\cong\bigoplus_{\gamma\in\Gamma}\Mat_{d_\gamma}(K)$$
    for a finite index set $\Gamma$.  Write $e=(e_\gamma)_{\gamma\in\Gamma}$ under this decomposition, and put $r_\gamma=\rank(e_\gamma)$.  For every $\gamma$,
    $$e_\gamma \Mat_{d_\gamma}(K)e_\gamma
   \cong
   \begin{cases}
      \Mat_{r_\gamma}(K),&r_\gamma>0,\\
      0,&r_\gamma=0.
   \end{cases}$$
   It follows that
   $$Z(e\qSj e)=\bigoplus_{\substack{\gamma\in\Gamma,\, r_\gamma>0}}K e_\gamma.$$
   On the other hand,
   $$Z(\qSj)=\bigoplus_{\gamma\in\Gamma}K1_\gamma,$$
where $1_\gamma$ is the identity of the $\gamma$-th matrix block, and therefore
$$eZ(\qSj)e
   =\bigoplus_{\substack{\gamma\in\Gamma,\, r_\gamma>0}}K e_\gamma
   =Z(e\qSj e).$$

\end{rmk}


\end{document}